\documentclass[12pt,a4paper]{amsart}
\usepackage{amsmath,amssymb,amsfonts}
\usepackage[all]{xy}
\usepackage{caption}
\usepackage{enumerate}
\usepackage{hyperref} 
\usepackage{bm}
\usepackage{graphicx}
\usepackage{xcolor}
\usepackage{multicol}
\usepackage{tabularx}
\usepackage[capitalize]{cleveref}

\usepackage[normalem]{ulem}

\usepackage{mathpazo}
\usepackage{hyperref}
\usepackage{appendix}
\usepackage{a4wide}

\theoremstyle{plain}
\newtheorem{thm}{Theorem}[section]

\newtheorem{lem}[thm]{Lemma}
\newtheorem{cor}[thm]{Corollary}
\newtheorem{prop}[thm]{Proposition}

\theoremstyle{definition}
\newtheorem{definition}[thm]{Definition}
\newtheorem{expl}[thm]{Example}

\newtheorem{question}[thm]{Question}

\newtheorem{rem}[thm]{Remark}

\makeatletter
\newcommand{\neutralize}[1]{\expandafter\let\csname c@#1\endcsname\count@}
\makeatother

\DeclareMathOperator{\R}{\mathbb{R}}
\DeclareMathOperator{\C}{\mathbb{C}}

\DeclareMathOperator{\N}{\mathbb{N}}

\DeclareMathOperator{\Spec}{\text{Spec}}

\DeclareMathOperator{\conc}{\mathbin{\ast}}

\DeclareMathOperator{\Tr}{\text{Tr}}

\newcommand{\one}{\mathbbold{1}}

    \DeclareFontFamily{U}{wncy}{}
    \DeclareFontShape{U}{wncy}{m}{n}{<->wncyr10}{}
    \DeclareSymbolFont{mcy}{U}{wncy}{m}{n}
    \DeclareMathSymbol{\Sha}{\mathord}{mcy}{"58}

\numberwithin{equation}{section}

\DeclareSymbolFont{bbold}{U}{bbold}{m}{n}
\DeclareSymbolFontAlphabet{\mathbbold}{bbold}
\newcommand{\ind}{\bm{1}}

\title{Joins of normal matrices, their spectrum, and applications}
  \author{J\'an Min\'a\v{c}, Lyle Muller,
  \\ Tung T. Nguyen, Federico W. Pasini}

\address{Department of Mathematics, Western University, London, Ontario, Canada N6A 5B7}
\email{minac@uwo.ca}

\address{Department of Mathematics, Western University, London, Ontario, Canada N6A 5B7}
\email{lmuller2@uwo.ca}

\address{Department of Mathematics, Western University, London, Ontario, Canada N6A 5B7}
 \email{tung.nguyen@uwo.ca}
 
\address{Huron University College, London Ontario, Canada N6A 5B7}
\email{fpasini@uwo.ca}
\date{\today}

\begin{document}
\maketitle
\begin{abstract}
Motivated by studies of oscillator networks, we study the spectrum of the join of several normal matrices with constant row sums. We apply our results to compute the characteristic polynomial of the join of several regular graphs. We then use this theorem to study several problems in spectral graph theory. In particular, we provide some simple constructions of Ramanujan graphs and give new proofs for some theorems in the classical book of Cvetkovi\'{c}, Rowlinson, and Slobodan. 
\end{abstract}

\section{Introduction}
Networks of nonlinear oscillators have attracted interest in several scientific domains such as theoretical physics, mathematical biology, power-grid systems, and many more. In our investigation of oscillator networks (see \cite{[KO3], [KO4], [CM1]}), the networks of several communities joined together often appear and provide some interesting phenomena (see for example \cite[Proposition 23]{[CM1]} and \cite[Proposition 12]{[KO4]}). The key idea of these investigations of multi-layer networks is to reduce the study of dynamics on complex networks to simpler networks. In both theory and practice, the adjacency matrix of the original multi-layer network may appear quite complicated. However, using the techniques that we develop here for join graphs, we will see that we can attach to each multi-layer network a reduced matrix which usually has a much smaller size than the original adjacency matrix. Nevertheless, as we show in \cite{nguyen2022broadcasting}, by considering this reduced matrix, we can obtain good information about the entire complex network.  In \cite{[CM1]}, we study the case where the connection within a community follows a simple rule, namely, each community is a circulant network. In this case, the main theorem in \cite{[CM1]}, which generalizes the Circulant Diagonalization Theorem (CDT), explicitly describes the spectrum of the joined network. In this article, we generalize this theorem to the case where each community forms a regular graph. This relaxation will allow us to investigate a broader class of networks. In particular, we are able to apply our generalized theorems to study several interesting problems in spectral graph theory.

We remark that since the completion of this article, we have utilized this circle of ideas to study some broadcasting and combining mechanisms on multi-layer networks of oscillators (see \cite{jain2023composed, nguyen2022broadcasting}.) We want to emphasize that the broadcasted solution described \cref{eq:broadcast} has a nice physical interpretation(see \cite{nguyen2022broadcasting}). Furthermore, we also extend this line of research further by investigating the question of whether a given graph can be written as a joined union of smaller graphs with a special focus on the case where the graph is highly symmetric (see \cite{chudnovsky2024prime, nguyen2024certain}).

\subsection{Outline}
The structure of this article is as follows. In \cref{sec:general}, we study some basic spectral properties of normal matrices with constant row sums. In \cref{sec:general_join}, we define the joins of these matrices and study their spectral properties. We then apply the main results from this section to give new proofs of several results in \cite{[CRS]} for the join of regular graphs. In this way, we provide a new conceptual insight for these statements based on key results in \cref{sec:general_join}. \cref{section:ramanujan_construction} explains a simple method to construct Ramanujan graphs using the join construction. We remark here that the construction of Ramanujan graphs is of great interest in network, communication, coding, and number theories.  We then discuss the joined union of graphs in \cref{sec:joined_union_general}. In \cref{sec:graph_energy}, we apply the results from the previous sections to study some questions in graph energy. In particular, we propose a question on the relation between the energy of several regular graphs and their joined union. Some notable results in this section are \cref{prop:first_estimate} and \cref{thm:second_main} where we provide some concrete evidence for our question. 

\begin{rem} 
We remark that a weaker form of \cref{thm:join_spectrum} has been discussed previously in \cite{[joined_union2], [joined_union]}. We refer the reader to  \cref{rem:difference} and  \cref{rem:semimagic} 
regarding our approach. We explain in these remarks why our approach is more flexible than \cite{[joined_union2], [joined_union]} and why we can apply main results in the situations where \cite{[joined_union2], [joined_union]} cannot be applied directly.
\end{rem}

\section{Normal matrices with constant row sums} \label{sec:general}

We start with a definition. 
\begin{definition} 
Let $A=(a_{ij})_{i,j}$ be an $n \times n$ matrix with complex coefficients. We say that $A$ is $r_A$-row regular if the sum of all entries in each row of $A$ is equal to $r_A$, namely
\[ \forall 1 \leq i \leq n,\;\sum_{j=1}^n a_{ij}=r_A.  \]
Similarly, we say that $A$ is $c_A$-column regular if the sum of all entries in each column of $A$ is equal to $c_A$. 
\end{definition} 

\begin{rem} \label{rem:magic_square}
Some authors use the term ``semimagic squares" for matrices that are both $r_A$-regular and $c_A$-regular and $r_A=c_A$ (see, for example \cite{[semimagic]}.) 

\end{rem}

We note that if $A$ is both $r_A$-row regular and $c_A$-column regular then $r_A=c_A$ as long as we work with matrices with coefficients in $\R$ or $\C$ (or more generally, over a field of characteristics $0$).  This can be seen by observing that the sum of all entries in $A$ is equal to both $nr_A$ and $nc_A$; and therefore $r_A=c_A$. Here is a simple criterion for row and column regularity.

\begin{lem} \label{lem:criterior}
Let $v=\one_n=(1, 1, \ldots, 1)^t \in \C^n$. Then $A$ is $r_A$-row regular if and only if $v$ is an eigenvector of $A$ associated with the eigenvalue $r_A$. 
Similarly, $A$ is $c_A$-column regular if and only if $v^{t}$ is a left eigenvector of $A$ associated with the eigenvalue $c_A.$
\end{lem} 
\begin{proof}
Obvious from the definition. 
\end{proof}

\begin{definition}
Let $A \in M_{n}(\C)$ be a matrix of size $n \times n$. We say that $A$
is normal if $AA^{*}=A^{*} A.$ Here $A^{*}$ is the conjugate transpose of $A$. 
\end{definition}

A special property of normal matrices is that they are always diagonalizable by an orthonormal basis of eigenvectors. 

\begin{thm} (see \cite[Theorem 2.5.3]{[HJ]})
Suppose $A$ is a normal matrix. Then its eigenspaces span $\C^n$ and are pairwise orthogonal with respect to the standard inner product on $\C^n.$
\end{thm}

A direct corollary of theorem is the following. 
\begin{cor} \label{cor:orthogonal}
Suppose that $A$ is both normal and $r_A$-row regular. Then there exists an orthonormal basis $\{v_1^{A}, v_2^{A}, \ldots, v_n^{A} \} $ of eigenvectors of $A$ associated with the eigenvalues $\{\lambda_{1}^A, \lambda_2^A, \ldots, \lambda_n^A \}$ such that 
\(v_1^{A}= \frac{1}{\sqrt{n}}\one_n=\frac{1}{\sqrt{n}}(1, \ldots, 1)^t \in \C^n\). 
In particular, $r_{A}=\lambda_{1}^A$ and, for $2 \leq k \leq n$, the standard inner product
\(\langle v_{1}^A, v_k^{A} \rangle =0 .\) 
\end{cor}

Another corollary is the following. 
\begin{cor}
If $A$ is both normal and $r_A$-row regular. Then $A$ is also $r_A$-column regular. In particular, $A$ is a semimagic square matrix. 
\end{cor}
\begin{proof}
Let $\{v_1^{A}, v_2^{A}, \ldots, v_n^{A} \} $ be the system of orthonormal eigenvectors of $A$ associated with the eigenvalues $\{r_A, \lambda_2^A, \ldots, \lambda_n^A \}$ as described in Corollary \ref{cor:orthogonal}. Let $V=(v_1^{A}, v_2^{A}, \ldots, v_n^A)$ be the $n \times n$ matrix formed by this system of eigenvectors and let $D=\text{diag}(r_{A}, \ldots, \lambda_n^A)$ be the digonal matrix of corresponding eigenvalues. We then have  $A V = VD$.   Since $\{v_1^A, v_2^A, \ldots, v_n^A \}$ is an orthonormal basis, we have $VV^{*}=V^{*}V=I_n$, and hence $V^{*}=V^{-1}$.  Therefore, we can rewrite the equation $AV=VD$ as 
\[ (V^{*}) A = D V^{*} .\] 
This shows that the rows of $V^{*}$, namely $\{(v_1^{A})^*, (v_2^A)^{*}, \ldots, (v_n^A)^{*} \}$ form a system of orthnormal \textit{left} eigenvectors for $A$ associated with the eigenvalues  $\{r_A, \lambda_2^A, \ldots, \lambda_n^A \}$. We conclude that the column sums of $A$ are equal to $\lambda_1^{A}=r_A$ as well. 

\end{proof}

\section{Joins of normal matrices with constant row sums} \label{sec:general_join}

Let $d, k_1, k_2, \ldots, k_d \in \N\setminus\{0\}$, and set $n=k_1+k_2+\ldots+k_d$.
Thus $\mathbf{k}_d=(k_1,\dots,k_d)$ is a partition of $n$ into $d$ non-zero summands. Following \cite{[CM1]}, we shall consider $n \times n$ matrices of the following form
\begin{equation}\label{eq:def of join}\tag{$\ast$}
A=\left(\begin{array}{c|c|c|c}
A_1 & a_{12}\ind & \cdots & a_{1d}\ind \\
\hline
a_{21}\ind & A_2 & \cdots & a_{2d}\ind \\
\hline
\vdots & \vdots & \ddots & \vdots \\
\hline
a_{d1}\ind & a_{d2}\ind & \cdots & A_d
\end{array}\right),
\end{equation}
where, for each $1 \leq i,j \leq d$, $A_i$ is a normal, $r_{A_i}$-row regular matrix of size $k_i \times k_i$ with complex entries, and $a_{i,j}\ind$ is a $k_i \times k_j$ matrix with all entries equal to a constant $a_{i,j}\in\mathbb{C}$.
These matrices will be called \textit{$\mathbf{k}_d$-joins of normal row regular} (\textit{NRR} for short) \textit{matrices}.

For each $1 \leq i \leq d$, let $\{v_1^{A_i}, v_2^{A_i}, \ldots, v_{k_i}^{A_i} \}$ and $\{\lambda_1^{A_i}, \lambda_2^{A_i}, \ldots, \lambda_{k_i}^{A_i} \}$ be the set of eigenvectors and eigenvalues of $A_i$ as described in Corollary \ref{cor:orthogonal}. The next proposition is a direct generalization of \cite[Proposition 10]{[CM1]}. Before stating it, let us introduce the convenient notation
\[
(x_1,\dots,x_m)^T \conc (y_1,\dots,y_n)^T = (x_1,\dots,x_m, y_1,\dots,y_n)^T.
\]
For more vectors, we can define $\conc$ inductively.

\begin{prop}\label{prop:circulant eigenvectors}
For each $1 \leq i \leq d$ and  $2 \leq j \leq k_i$ let 
\[ \begin{aligned}
w_{i,j}&=\vec{0}_{k_1} \conc \ldots \conc\vec{0}_{k_{i-1}}\conc v_{j}^{A_i} \conc \vec{0}_{k_{i+1}} \conc \ldots \conc \vec{0}_{k_d} .\\
\end{aligned}
\] 
Then $w_{i,j}$ is an eigenvector of $A$ associated with the eigenvalue $\lambda_{j}^{A_i}.$
\end{prop} 
\begin{proof}
By direct inspection, the key property being that, for $1\leq\ell\leq d,\;\ell\neq i$ and $2\leq j\leq k_i$, $\langle a_{\ell,i}\one_{k_i},v^{A_i}_j\rangle=0$, according to Corollary \ref{cor:orthogonal}.

\end{proof}
We will refer to the $w_{i,j}$'s and to the associated eigenvalues $\lambda_{j}^{A_i}$ as the old NRR eigenvectors and eigenvalues of $A$. Let $\lambda_1, \lambda_2, \ldots \lambda_d$ be the (not necessarily distinct) remaining eigenvalues of $A$.
\begin{definition}
The reduced characteristic polynomial of $A$ is 
\[ \overline{p}_{A}(t)=\prod_{i=1}^d (t-\lambda_i)=\dfrac{p_{A}(t)}{\prod_{\substack{1 \leq i \leq d, \\ 2 \leq  j \leq k_i}} (t-\lambda_{j}^{A_i})} =\frac{p_A(t)}{\prod_{i=1}^d \frac{p_{A_i}(t)}{t-r_{A_i}}} .\] 
\end{definition} 
We will now describe $\overline{p}_{A}(t)$ as the characteristic polynomial of the matrix
\[ \overline A= 
\begin{pmatrix}
r_{A_1} & a_{12}k_2 & \cdots & a_{1d}k_d \\
a_{21}k_1 & r_{A_2} & \cdots & a_{2d}k_d \\
\vdots  & \vdots  & \ddots & \vdots  \\
a_{d1}k_1 & a_{d2}k_2 & \cdots & r_{A_d} 
\end{pmatrix}.\] 

For a vector $w=(x_1,\dots,x_d)\in\C^d$, we define 
\begin{equation} \label{eq:broadcast}
w^\otimes=(\underbrace{x_1, \ldots, x_1}_{\text{$k_1$ terms}},\dots,\underbrace{x_d, \ldots, x_d}_{\text{$k_d$ terms}})^{t}\in \C^n
\end{equation}

\begin{thm} \label{thm:join_spectrum}
The reduced characteristic polynomial of $A$ coincides with the characteristic polynomial of $\overline A$, namely
\[ \overline{p}_{A}(t)=p_{\overline A}(t) .\] 
In other words 
\[ p_{A}(t)= p_{\overline{A}}(t) {\prod_{\substack{1 \leq i \leq d, \\ 2 \leq  j \leq k_i}} (t-\lambda_{j}^{A_i})} .\] 
\end{thm}
\begin{proof}
Firstly, we note that, by construction, for any $v\in\C^d$ and any $\lambda\in\C$
\begin{equation}\label{eq:tensor expansion}
\left[(\overline A-\lambda I)v\right]^\otimes = (A-\lambda I)v^\otimes.
\end{equation}
Let ${\lambda}$ be an eigenvalue of $\overline{A}$, and let ${w}=(x_1,\dots,x_d)$ be an associated generalized eigenvector, satisfying $(\overline{A}-{\lambda}I_d)^m{w}=0$ for a suitable $m$. 
We will show, by induction on $m$, that $(A-\lambda I_n)^mw^\otimes=0$.
If $m=1$, the assertion is a consequence of Equation \eqref{eq:tensor expansion}. If $m>1$, consider the vector $w'=(\overline{A}-\lambda I_d)w$, which satisfies $(\overline{A}-{\lambda}I_d)^{m-1}{w'}=0$. By induction hypothesis, $({A}-{\lambda}I_n)^{m-1}{(w')}^\otimes=0$, therefore, thanks to Equation \eqref{eq:tensor expansion},
\[
(A-\lambda I_n)^mw^\otimes=(A-\lambda I_n)^{m-1}\left((A-\lambda I_n)w^\otimes\right)=(A-\lambda I)^{m-1}(w')^\otimes=0.
\]
In other words, the generalized eigenspaces of $\overline{A}$ lift to (direct summands of) generalized eigenspaces of $A$.
Now we observe that the NRR eigenvectors of $A$, together with the generalized eigenvectors of $A$ of the shape $w^\otimes$, $w\in\C^d$, form a linearly independent set thanks to Corollary \ref{cor:orthogonal}. Hence, by dimension counting, the eigenvalues of $\overline{A}$ are precisely the eigenvalues $\lambda_1,\dots,\lambda_d$ of $A$, with the correct multiplicity. Equivalently, $\overline{p}_A(t)=p_{\overline{A}}(t)$.
\end{proof}
\begin{rem} \label{rem:difference}
After proving Theorem \ref{thm:join_spectrum}, we learned from ResearchGate that a special form of this theorem has been proved in \cite[Theorem 2.1]{[joined_union]} and \cite[Theorem 3]{[joined_union2]}. We would like to take this chance to clarify the similarities and differences between our approaches. First, both our methods investigate the ``broadcasting'' mechanism to lift eigenvalues and eigenvectors from $\bar{A}$ to $A$ as described by \cref{eq:broadcast} (this broadcasting procedure has a physical interpretation as we explained in our work \cite{nguyen2022broadcasting}.) If $A$ is symmetric, then $A$ is diagonalizable and hence $\bar{A}$ is diagonalizable as well. In this case, \cite[Theorem 2.1]{[joined_union]} only needs to deal with eigenvalues. Our method shows that the broadcasted solution described by \cref{eq:broadcast} even works at the level of \textit{generalized} eigenvalues. In other words, it even works for the cases where either $\overline{A}$ is not diagonalizable or $A$ is not a symmetric matrix. This is important for applications because many networks in the Kuramoto models are directed.
\end{rem}

\begin{rem} \label{rem:semimagic}
We discuss a generalization of Theorem \ref{thm:join_spectrum}. More precisely, we can show that \cref{thm:join_spectrum} holds for any field $F$ under the mild assumption that $k_i$ is invertible in $F$ for all $1 \leq i \leq d$. In particular, we can drop the ``normal'' condition on $A_i$. First, we recall from Remark \ref{rem:magic_square} that a $k_1 \times k_1$ matrix $A_1$ with entries in a field $F$ is called a semimagic square if $A_1$ is both $r_{A_1}$-regular and $c_{A_1}$-regular and $c_{A_1}=r_{A_1}.$ If $k_1$ is invertible in $F$, then $F^{k_1}$ can be decomposed into 
\begin{equation} \label{eq:decomposition}
     F^{k_1}= F \one_{k_1} \oplus W_1 .
\end{equation}
     
Here $F\one_{k_1}$ is the one dimensional vector space generated by $\one_{k_1}$ and $W_1$ is the set of all vectors $(x_1, x_2, \ldots, x_{k_1}) \in F^{k_1}$ such that $ \sum_{i=1}^{k_1} x_i =0$. We can check that each component of this decomposition is stable under $A_1$ for any semimagic square $A_1$. Now suppose that $A$ is the join of $d$ semimagic squares $A_i$ of sizes $k_i \times k_i$ as defined in equation \ref{eq:def of join}. We assume that further that $k_i$ is invertible in the field $F$. Let $W_i$ be the decomposition 
\[ F^{k_i}=F \one_{k_i} \oplus W_i .\] 
We see that for $1 \leq i \leq d$
\[ \widehat{W_i}= \{\vec{0}_{k_1} \conc \ldots \conc\vec{0}_{k_{i-1}}\conc v_{i} \conc \vec{0}_{k_{i+1}} \conc \ldots \conc \vec{0}_{k_d} \quad | v_i \in W_i \}, \] 
is an $A$-stable subpsace of $F^{k_i}$. By the same proof as explained in Theorem \ref{thm:join_spectrum}, we can see that 
\begin{equation} \label{eq:generalization}
\overline{p}_{A}(t)=p_{\overline A}(t) .
\end{equation}
We also note that the set of all such $A$ with coefficients in any ring $R$ has the structure of a ring (the case $d=1$ was considered in \cite{[semimagic]}). By the same method described in the proof of \cite[Theorem 3.16]{[joint_group_ring]},  we could describe the structure of this ring and derive Equation \ref{eq:generalization} as a direct consequence. We could show, in particular, that the map $A \mapsto \bar{A}$ is a ring homomorphism. 
\end{rem}

\section{Applications to spectral graph theory}
\subsection{Spectrum of the join of regular graphs} \label{section:graph_join}
In this section, we apply Theorem \ref{thm:join_spectrum} to give new proofs for Theorem 2.1.8 and Theorem 2.1.9 in \cite{[CRS]}. Let $G_1, G_2, \ldots, G_d$ be undirected regular graphs such that $G_i$ has degree $r_i$ and $k_i$ vertices. Let $G$ be the join graph of $G_1, G_2, \ldots, G_d$, which we will denote by $G=G_1 + G_2 + \ldots +G_d$. We recall that $G$ is obtained from the disjoint union of $G_1, \ldots, G_2, \ldots, G_d$ by joining each vertex $G_i$ with each vertex in all others $G_j$ for $j \neq i$ (see \cite[Section 4]{[CM1]} and the reference therein for further details). Let $A_i$ be the adjacency matrix of $G_i$ for $1 \leq i \leq d$ and $A$ be the adjacency matrix of $G$. By definition of the join of graphs, the adjacency matrix $A$ of $G$ has the following form 
\[
A=\left(\begin{array}{c|c|c|c}
A_1 & \ind & \cdots & \ind \\
\hline
\ind & A_2 & \cdots & \ind \\
\hline
\vdots & \vdots & \ddots & \vdots \\
\hline
\ind & \ind & \cdots & A_d
\end{array}\right).
\]
Since $G_i$ is an undirected graph, $A_i$ is real and symmetric, hence normal. Furthermore, since $G_i$ is regular of degree $r_i$, $A_i$ is $r_i$-row regular. By Theorem \ref{thm:join_spectrum}, the reduced characteristics polynomial of $A$ is given by
\[ \overline{p}_{A}(t)=p_{\overline A}(t) ,\]
where 

\[ \overline A= 
\begin{pmatrix}
r_{1} & k_2 & \cdots & k_d \\
k_1 & r_{2} & \cdots & k_d \\
\vdots  & \vdots  & \ddots & \vdots  \\
k_1 & k_2 & \cdots & r_{d} 
\end{pmatrix} .\] 
In summary, we have
\begin{prop} \label{prop:join_spetrum}
The characteristic polynomial of $A$ is given by 
\[ p_{A}(t)=p_{\overline{A}}(t)  \dfrac{\prod_{i=1}^d p_{A_i}(t)}{\prod_{i=1}^d  (t-r_i)}.\] 
\end{prop}

Let us consider some special cases of this proposition. 
\begin{cor}(See \cite[Theorem 2.1.8]{[CRS]})
If $G_1$ is $r_1$-regular with $k_1$ vertices and $G_2$ is $r_2$-regular with $k_2$ vertices then the characteristic polynomial of the join $G_1+G_2$ is given by 
\[ p_{G_1+G_2}(t)=\frac{p_{G_1}(t) p_{G_2}(t)}{(t-r_1)(t-r_2)} \left((t-r_1)(t-r_2)-k_1k_2 \right) .\] 
\end{cor}
\begin{proof}
Let $A_1, A_2$ be the adjacency matrix of $G_1, G_2$ respectively. Then, the adjacency matrix of $G_1+G_2$ is 
\[ A= \begin{pmatrix}
A_1 & \ind \\ \ind & A_2 
\end{pmatrix} .\] 
We have 
\[ \overline{A}= \begin{pmatrix}
r_1 & k_2 \\ k_1 & r_2 
\end{pmatrix} .\]
Hence 
\[ p_{\overline{A}}(t)=(t-r_1)(t-r_2)-k_1k_2 .\] 
By Proposition \ref{prop:circulant eigenvectors}, we conclude that

\[ p_{G_1+G_2}(t)=\frac{p_{G_1}(t) p_{G_2}(t)}{(t-r_1)(t-r_2)} \left((t-r_1)(t-r_2)-k_1k_2 \right) .\] 

\end{proof}

\begin{cor}(See \cite[Theorem 2.1.9]{[CRS]} \label{cor:equal_difference}
Let $G_i$ be $r_i$-regular with $k_i$ vertices. Assume further that 
\[ k_1-r_1=k_2-r_2=\ldots=k_d-r_d=s .\] 
Let $G$ be the join graph of $G_1, G_2, \ldots, G_d$. Let \[ n=k_1+k_2+\ldots+k_d ,\] 
and 
\[ r=n-s .\]
Then 
\begin{enumerate}
    \item $G$ is $r$-regular with $n$ vertices. 
    \item The characteristic polynomial of $G$ is given by 
\[ p_{G}(t)=(x-r)(x+s)^{d-1} \dfrac{\prod_{i=1}^d p_{G_i}(t)}{\prod_{i=1}^d  (t-r_i)}. \]
\end{enumerate}
\begin{proof}
Let $v_i$ be a vertex in $G_i$. By definition, the degree of $v_i$ in $G$ is given by 
\[ \deg_{G_i}(v_i)+(n-k_i)= n-(k_i-r_i)=n-s=r.\] 
We conclude that $G$ is $r$-regular. This proves part $(1)$. For part $(2)$, we note that if $A$ is the adjacency matrix of $G$ then $\overline{A}$ is given by
\[ \overline A= 
\begin{pmatrix}
r_1 & k_2 & \cdots & k_d \\
k_1 & r_2 & \cdots & k_d \\
\vdots  & \vdots  & \ddots & \vdots  \\
k_1 & k_2 & \cdots & r_d 
\end{pmatrix}.\] 
We observe that 
\[ \overline{A}+sI_{d}= \begin{pmatrix}
k_1 & k_2 & \cdots & k_d \\
k_1 & k_2 & \cdots & k_d \\
\vdots  & \vdots  & \ddots & \vdots  \\
k_1 & k_2 & \cdots & k_d 
\end{pmatrix}\] 
has rank $1$. Consequently, $-s$ is an eigenvalue of $\overline{A}$ with multiplicity at least $d-1$. Additionally, by part $(1)$, $G$ is $r$-regular, hence $\lambda=r$ is the remaining eigenvalue of $\overline{A}$. Consequently,
\[ p_{\overline{A}}(t)=(t-r)(t+s)^{d-1} .\] 
By Proposition \ref{prop:join_spetrum}, we conclude that 
\[ p_{G}(t)=(t-r)(t+s)^{d-1} \dfrac{\prod_{i=1}^d p_{G_i}(t)}{\prod_{i=1}^d  (t-r_i)}. \]
\end{proof}

\end{cor}
\subsection{A simple construction of Ramanujan graphs} \label{section:ramanujan_construction}

We discuss some applications of Corollary \ref{cor:equal_difference} to the construction of Ramanujan graphs. We first recall the definition of these graphs (see \cite[Chapter 3]{[CRS]} and \cite{[Murty]} for further details.) We also recommend \cite{hoory2006expander} for a beautiful survey of some surprising applications and occurrence of Ramanujan graphs in various parts of mathematics, physics, communications networks and computer science.) 
\begin{definition} (see \cite[Definition 3.5.4]{[CRS]})
Let $G$ be a connected $r$-regular graph with $k$ vertices, and 
let $r=\displaystyle \lambda _{1}\geq \lambda_{2}\geq \cdots \geq \lambda_{n}$ be the eigenvalues of the adjacency matrix of $G$.  Since $G$ is connected and $r$-regular, its eigenvalues satisfy 
\( |\lambda_i| \leq r, 1 \leq i \leq n.\)
Let 
\[\lambda (G)=\max_{|\lambda_i|<r}|\lambda_{i}| .\]
The graph $G$ is a \textit{Ramanujan graph} if  
\[ \lambda (G)\leq 2{\sqrt {r -1}} .\] 
\end{definition}

The following proposition provides a construction of Ramanujan graphs.
\begin{prop} \label{prop:rm1}
Let $d\geq 2$ and, for $1\leq i\leq d$, let $G_i$ be $r_i$-regular Ramanujan graphs with $k_i$ vertices. Suppose further that the $G_i$'s satisfy the same conditions as in Corollary \ref{cor:equal_difference}, namely  
\[ k_1-r_1=k_2-r_2=\ldots=k_d-r_d=s .\] 
Let $G$ be the join graph of $G_1, G_2, \ldots, G_d$ and
\( n=k_1+ k_2+ \ldots+k_d\). 
Then $G$ is a Ramanujan graph if and only if 
\[ s \leq 2 (\sqrt{n}-1) .\] 
\end{prop}
\begin{proof}
Corollary \ref{cor:equal_difference} describes the eigenvalues of $G$. Taking into account that the valency $r$ of $G$ is greater than the valency $r_i$ of each $G_i$, and that each $G_i$ is Ramanujan, $G$ is Ramanujan if and only if $s\leq 2\sqrt{r-1}=2\sqrt{n-s-1}$, if and only if $s^2+4s-4n+4\leq 0$, if and only if $s\leq 2\sqrt{n}-2$.
\end{proof}
Here is a special case of this construction. 
\begin{cor} \label{prop:Ramanujan_1}
Let $G$ be a $r$-regular graph with $k$ vertices. Let $G^{d}$ be the join graph of $d$ identical copies of $G$. Then there exists a natural number $d_0$ such that for all $d \geq d_0$, $G^d$ is a Ramanujan graph. 

\end{cor}
\begin{proof}
By Proposition \ref{prop:rm1}, $G^d$ is a Ramanujan graph if and only if 
\[ k-r \leq 2(\sqrt{dk}-1) .\] 
This is equivalent to 
\[ d \geq \frac{1}{k} \left(\frac{k-r}{2}+1 \right)^2 .\] 
We therefore can take 
\[ d_0= \left\lceil \frac{1}{k} \left(\frac{k-r}{2}+1 \right)^2 \right\rceil .\] 

\end{proof}

\subsection{Spectrum of the joined union of graphs} \label{sec:joined_union_general}

Let $G$ be a graph with $d$ vertices $\{v_1, v_2, \ldots, v_d\}$. Let $G_1, G_2, \ldots, G_d$ be graphs. The joined union $G[G_1, G_2, \ldots, G_d]$ is obtained from the union of $G_1, \ldots, G_d$ by joining with an edge each pair of a vertex from $G_i$ and a vertex from $G_j$ whenever $v_i$ and $v_j$ are adjacent in $G$ (see \cite{[joined_union]} for further details). Let $A_{G} =(a_{ij})$ be the adjacency matrix of $G$ and $A_{1}, A_{2}, \ldots, A_{d}$ be the adjacency matrices of $G_1, G_2, \ldots, G_d$ respectively. The adjacency matrix of $G[G_1, G_2, \ldots, G_d]$ has the following form 
\begin{equation} \label{eq:matrix_A}
A=\left(\begin{array}{c|c|c|c}
A_1 & a_{12}\ind & \cdots & a_{1d}\ind \\
\hline
a_{21}\ind & A_2 & \cdots & a_{2d}\ind \\
\hline
\vdots & \vdots & \ddots & \vdots \\
\hline
a_{d1}\ind & a_{d2}\ind & \cdots & A_d
\end{array}\right).
\end{equation}

\begin{rem} \label{rem:special_case}
When $G=K_d$, the complete graph on $d$ vertices, $G[G_1, G_2, \ldots, G_d]$ is exactly the join graph of $G_1, G_2, \ldots, G_d$ discussed in  \cref{section:graph_join}. 
\end{rem}

By  \cref{thm:join_spectrum}, the spectrum of $G[G_1, G_2, \ldots, G_d]$ can be described by the spectra of $G_i$ and an auxiliary matrix describing the interconnections between $G_i$. More precisely, we have the following proposition.
\begin{prop} \label{prop:joined_union}
Assume that for each $1 \leq i \leq d$, $G_i$ is a $r_i$-regular graph with $k_i$ nodes. Let $G[G_1, G_2, \ldots, G_d]$ be the joined union graph. Let $\{\lambda_{1}^{G_i}=r_i, \ldots, \lambda_{k_i}^{G_i} \}$ be the spectrum of $G_i$ as described in Corollary \ref{cor:orthogonal}. Then the spectrum of $A$ is the union of $\Spec(\overline{A})$ and the following multiset 
\[ \{\lambda_j^{A_i} \}_{1 \leq i \leq d, 2 \leq j \leq k_i} .\] 

Here $\overline{A}$ is the following $d \times d$ matrix, whose entries are the row sums of the blocks in the matrix $A$ 
\[ \overline A= 
\begin{pmatrix}
r_{A_1} & a_{12}k_2 & \cdots & a_{1n}k_d \\
a_{21}k_1 & r_{A_2} & \cdots & a_{2n}k_d \\
\vdots  & \vdots  & \ddots & \vdots  \\
a_{d1}k_1 & a_{d2}k_2 & \cdots & r_{A_d} 
\end{pmatrix} .\] 

\end{prop}
\begin{proof}

    This proposition follows from \cref{thm:join_spectrum}. To see this, we recall that the adjacency matrix of $G[G_1, G_2, \ldots, G_d]$ has the form described in \cref{eq:matrix_A} where $A_G$ is the adjacency matrix of $G$ and $A_1, A_2, \ldots, A_d$ are the adjacency matrices of $G_1, G_2, \ldots, G_d$ respectively. Since $G_i$ is an undirected graph, we know that $A_i$ is symmetric. Furthermore, by our assumption, $G_i$ is $r_i$-regular, $A_i$ is a normal, row regular with $r_i = r_{A_i}.$ Therefore, we can apply \cref{thm:join_spectrum} to obtain the above description for the spectrum of $A.$
\end{proof}

Let us consider another special case where the $G_i$ are all $r$-regular graphs with $k$ vertices. In this case, we have $k_1= k_2 = \ldots =k_d=k$ and $r_{A_1}=r_{A_2}=\ldots = r_{A_d} =r$. Therefore, by \cref{prop:joined_union}, we have the following.

\begin{prop} \label{prop:joined_union_equal_case}
Assume that for each $1 \leq i \leq d$, $G_i$ is a $r$-regular graph with $k$ vertices. Let $G[G_1, G_2, \ldots, G_d]$ be the joined union graph. Let $\{\lambda_{1}^{G_i}=r, \ldots, \lambda_{k_i}^{G_i} \}$ be the spectrum of $G_i$ as described in Corollary \ref{cor:orthogonal}. Then the spectrum of $A$ is the union the multiset 
\[ \{\lambda_j^{A_i} \}_{1 \leq i \leq d, 2 \leq j \leq k_i} ,\] 
and the following multiset
\[ \{r +k \sigma| \sigma \in \Spec(A_G) \}.\] 
\end{prop}
\begin{proof}
In this case, the matrix $\overline{A}$ is of the following form
\[ \overline{A}= r I_d +k A_{G}, \]
where $A_G$ is the adjacency matrix of $G$. Thus the spectrum of $\overline{A}$ consists of the roots of the characteristic polynomial 
\[ p_{\overline{A}}(t) = \det(tI_d-rI_d-kA_G).\]
Therefore, the spectrum of $\overline{A}$ is given by $r+k \Spec(A_G)$. 
\end{proof}

\subsection{Energy of the joined union of graphs}  \label{sec:graph_energy}
The concept of graph energy originates from problems in theoretical chemistry. Specifically, the mathematical definition of graph energy was inspired by early studies on the total $\pi$-electron energy of molecules represented by molecular graphs (see \cite{coulson1940calculation, mcclelland1971properties, gutman1977acyclic}). Interest in graph energy remained relatively dormant until around 2000, when a small group of mathematicians mutually found their interest in this topic, leading to an explosion of research. For a more detailed discussion on the historical development of graph energy, we refer the reader to the survey article \cite{gutman2019graph}.

Our interest in this topic arises from our experimental observation that the graph energy seems to increase when we apply the join operation on graphs. The goal of this section is to formalize this observation and propose a precise question about the relationship between the energy of the joined union of graphs and the energy of individual graphs (see \cref{question:inequality}).

We first recall the definition of energy of a graph. 
\begin{definition} 
Let $G$ be a graph with $d$ nodes. Suppose that 
\[ \Spec(G)=\{\lambda_1, \lambda_2, \ldots, \lambda_d \} .\] 
The energy of $G$ is defined to be the following sum (see \cite[Section 9.2.2]{[CRS]} for further discussions.)
\[ E(G)=\sum_{i=1}^d |\lambda_i| .\] 
\end{definition}
\begin{expl}
If $G=K_d$ the complete graph with $d$ vertices. Then 
\[ \Spec(G)=\{[-1]_{d-1}, [d-1]_{1} \}, \] 
where $[a]_m$ means that $a$ has multiplicity $m.$ We conclude that the energy of $K_d$ is $2(d-1).$

\end{expl}

Let $G_i$ and $G$ be as at the beginning of \cref{section:graph_join}, namely \[ G=G_1+G_2+ \ldots+G_d=K_d[G_1, G_2, \ldots, G_d]. \]

We have the following inequality. 
\begin{thm} \label{prop:first_estimate}
The energy of $G$ is strictly larger than the sum of the energy of $G_i$: 
\[ E(G)> \sum_{i=1}^d E(G_i) .\] 
\end{thm}
\begin{proof}
Let $\{\lambda_1, \lambda_2, \ldots, \lambda_d \}$ be the eigenvalues of $\overline{A}$ where $A$ and $\overline{A}$ are the matrices defined at the beginning of \cref{section:graph_join}, namely 
\[ \overline A= 
\begin{pmatrix}
r_1 & k_2 & \cdots & k_d \\
k_1 & r_2 & \cdots & k_d \\
\vdots  & \vdots  & \ddots & \vdots  \\
k_1 & k_2 & \cdots & r_d 
\end{pmatrix} .\] 
Note that $\lambda_i \in \R$ as they are also eigenvalues of $A$, which is real and symmetric. By Proposition \ref{prop:join_spetrum}, we have 
\[ E(G)-\sum_{i=1}^{d} E(G_i)=\sum_{i=1}^{d} |\lambda_d|- \sum_{i=1}^{d} r_i .\] 
We also note that $\sum_{i=1}^d \lambda_i=\Tr(\overline{A})=\sum_{i=1}^d r_i.$ Therefore, we have 
\[ E(G)-\sum_{i=1}^{d} E(G_i)= \sum_{i=1}^{d} (|\lambda_i|-\lambda_i)=2 \sum_{ \lambda_i<0} |\lambda_i|  .\] 
Hence, to show that $E(G)>\sum_{i=1}^d E(G_i)$, we only need to show that for some $i$, $\lambda_i<0.$

Let $s_i=k_i-r_i>0$. Without loss of generality, we can assume that 
\[ k_1-r_1 \leq k_2-r_2 \leq \ldots \leq k_d-r_d .\] 
Let us consider 
\begin{align*} p_{\overline{A}}(-s_1)&= p_{\overline{A}}(r_1-k_1)= \det((r_1-k_1)-\overline{A}) \\
&=(-1)^d \det \begin{pmatrix}
k_1 & k_2 & \cdots & k_d \\
k_1 & r_2+k_1-r_1 & \cdots & k_d \\
\vdots  & \vdots  & \ddots & \vdots  \\
k_1 & k_2 & \cdots & r_d+k_1-r_1
\end{pmatrix} \\
&= (-1)^d k_1 \det \begin{pmatrix}
1 & k_2 & \cdots & k_d \\
1 & r_2+k_1-r_1 & \cdots & k_d \\
\vdots  & \vdots  & \ddots & \vdots  \\
1 & k_2 & \cdots & r_d+k_1-r_1
\end{pmatrix}.
\end{align*} 
By adding $-k_i$ times the first column to the $i$-th column, we see that the later determinant is also equal to 

\[ \det \begin{pmatrix}
1 & 0 & \cdots & 0 \\
1 & (k_1-r_1)-(k_2-r_2) & \cdots & 0 \\
\vdots  & \vdots  & \ddots & \vdots  \\
1 & 0 & \cdots & (k_1-r_1)-(k_d-r_d)
\end{pmatrix}=(s_1-s_2)(s_1-s_3) \ldots (s_1-s_d) .\] 
We conclude that  
\[ p_{\overline{A}}(-s_1)=(-1)^d k_1 \prod_{j \neq 1} (s_1-s_j)=-k_1 \prod_{j \neq 1}(s_j-s_1) \leq 0 .\] 
By the same argument, we see that 
\[ p_{\overline{A}}(-s_2)=-k_2 \prod_{j \neq 2} (s_j-s_2)=k_2(s_2-s_1) \prod_{j>2} (s_j-s_2) \geq 0 .\] 
By the mean value theorem, $p_{\overline{A}}(t)$ has a real root on the interval $[-s_2, -s_1]$. In particular, at least one eigenvalue of $\overline{A}$ must be negative. This completes the proof. 
\end{proof}

\begin{definition}
A graph $G$ with $d$ nodes is called hyperenergetic if $E(G) \geq 2(d-1).$
\end{definition}

\begin{thm} \label{thm:second_main}
Assume that $G_i$ are all $r$-regular with $k$ vertices. Assume further that $G$ is hyperenergetic. Then 
\[ E(G[G_1, G_2, \ldots, G_d]) \geq E(G)+\sum_{i=1}^d E(G_i) .\] 
The equality can happen, for example when $G$ and $G_i$ are all complete graphs. 
\end{thm}
\begin{proof}
Let $A$ be the adjacency matrix of $G[G_1, G_2, \ldots, G_d]$. Then the matrix $\overline{A}$ in Proposition \ref{prop:joined_union} has the following form 

\[ \overline A= 
\begin{pmatrix}
r & a_{12}k & \cdots & a_{1n}k \\
a_{21}k & r & \cdots & a_{2n}k \\
\vdots  & \vdots  & \ddots & \vdots  \\
a_{d1}k & a_{d2}k & \cdots & r 
\end{pmatrix}=rI_d+ k A_{G} .\] 
Let $\Spec(A_{G})=\{\lambda_1, \lambda_2, \ldots, \lambda_d \}$ then
\[ \Spec(\overline{A})=\{r+k \lambda_1, r+k \lambda_2, \ldots, r+k \lambda_d \}. \] 
By Proposition \ref{prop:joined_union}, we have
\begin{align*}
 E(G[G_1, G_2, \ldots, G_d])- E(G)-\sum_{i=1}^d E(G_i) &=\sum_{i=1}^{d}|r+k \lambda_i|- \sum_{i=1}^{d}|\lambda_i|- dr. 
\end{align*} 
We note that by the Perron-Frobenius Theorem, one of the eigenvalues of $A_{G}$ must be real and non-negative. Let us assume $\lambda_1 \geq 0.$ We then have
\begin{align*}
\sum_{i=1}^d |r+k \lambda_i|&=r+k \lambda_1+ \sum_{i=2}^d |r+k \lambda_i|\\
  & \geq r+k \lambda_1+\sum_{i=2}^d (k|\lambda_i|-r) \\
  &\geq k\sum_{i=1}^d |\lambda_i| -(d-2)r.
\end{align*} 
Consequently, we have 
\begin{align*}
 E(G[G_1, G_2, \ldots, G_d])- E(G)-\sum_{i=1}^d E(G_i) & \geq (k-1) \sum_{i=1}^d |\lambda_i|-2(d-1)r \\ 
 & \geq r (\sum_{i=1}^d |\lambda_i|-2(d-1))\\
 & \geq 0.
\end{align*}

Note that the second inequality follows from $k \geq r+1$ and the last inequality follows from the assumption that $G$ is hyperenergetic. 
\end{proof}

\begin{rem}
The above proof can be slightly generalized as follows. Suppose that $G$ is an undirected graph and the spectrum of $G$ consists of $n$ negative eigenvalues and $p$ non-negative eigenvalues. Suppose that the energy of $G$ satisfies 
\begin{equation} \label{eq:inequality1}
E(G) \geq d+n-p =2(d-p). 
\end{equation}
Then we have 
\[ E(G[G_1, G_2, \ldots, G_d]) \geq E(G)+\sum_{i=1}^d E(G_i) .\] 

We checked that all undirected graphs with at most $3$ nodes satisfy the Inequality \ref{eq:inequality1}.

\end{rem}

\begin{question} \label{question:inequality}
Suppose that $G_i$ are all regular graphs. Does the following inequality hold in general? 
\begin{equation} \label{eq:inequality}
E(G[G_1, G_2, \ldots, G_d]) \geq E(G)+\sum_{i=1}^d E(G_i)?
\end{equation}
\end{question}
We provide an answer to this question in a special case, namely for $d=2$. 
\begin{prop}
Let $G_1, G_2$ be two regular graphs and $G$ be a graph with $2$ nodes. Then \[ E(G[G_1, G_2]) \geq E(G)+ E(G_1)+E(G_2) .\] 
\end{prop}
\begin{proof}
If $G$ is the cocomplete graph, we have 
\[ E(G[G_1, G_2]) =  E(G)+ E(G_1)+E(G_2) .\] 
Suppose now that $G=K_2$ the complete graph on $2$ nodes. The energy of $G$ is $E(G)=2.$ Suppose that $G_i$ is $r_i$ regular with $k_i$ vertices for $i \in \{1, 2 \}$. Let $\lambda_1, \lambda_2$ be the eigenvalues of $\overline{A}$ where 
\[ \overline{A}= \begin{pmatrix} r_1 & k_2 \\ k_1 & r_2 \end{pmatrix} .\] 
By Proposition \ref{prop:join_spetrum} we have 
\[ E(G[G_1, G_2])-E(G_1)-E(G_2)=|\lambda_1|+|\lambda_2|-(r_1+r_2) .\] 
We conclude that 
\[ \lambda_1, \lambda_2= \frac{(r_1+r_2) \pm \sqrt{(r_1-r_2)^2+4k_1 k_2}}{2} .\] 

We have $\det(\overline{A})=r_1r_2-k_1 k_2<0$ so one root of $\overline{A}$ is negative and the other is positive. Consequently 
\begin{align*}
|\lambda_1|+|\lambda_2|-r_1-r_2 &=\sqrt{(r_1-r_2)^2+4k_1k_2}-(r_1+r_2) \\ &\geq \sqrt{(r_1-r_2)^2+4(r_1+1)(r_2+1)}-(r_1+r_2) \\ & \geq (r_1+r_2+2)-(r_1+r_2)=2 .
\end{align*} 
In other words, we have 
\[ E(G[G_1, G_2]) \geq E(G)+ E(G_1)+E(G_2) .\] 
\end{proof}

Another situation where we can verify Inequality \ref{eq:inequality} is the following. 
\begin{prop}
Let $G_i$ be $r_i$-regular with $k_i$ vertices. Assume further that 
\[ k_1-r_1=k_2-r_2=\ldots=k_d-r_d=s .\] 
Let $G$ be the joined union graph $K_d[G_1, G_2, \ldots, G_d]$. Then 

\[ E(K_d[G_1, G_2, \ldots, G_d]) \geq E(K_d)+\sum_{i=1}^d E(G_i). \] 
\end{prop}
\begin{proof} 
Let $k=\sum_{i=1}^d k_i$. By Corollary \ref{cor:equal_difference}, we have 
\begin{align*}
& E(K_d[G_1, G_2, \ldots, G_d])- E(K_d)-\sum_{i=1}^d E(G_i) \\ &=(k-s)+(d-1)s-2(d-1)-\sum_{i=1}^d r_i \\ 
&=\sum_{i=1}^d (k_i-r_i)-s+(d-1)(s-2) \\
&= ds-s+(d-1)(s-2) = 2(d-1)(s-1) \geq 0.
\end{align*}
Consequently 
\[ E(K_d[G_1, G_2, \ldots, G_d]) \geq E(K_d)+\sum_{i=1}^d E(G_i). \] 
\end{proof} 

\begin{prop}
Let $G_i$ be $r_i$-regular with $k_i$ vertices. Let $s_i=k_i-r_i$ Assume further that 
\[ s_1 <s_2<\ldots <s_d .\] 
Let $G$ be the joined union graph $K_d[G_1, G_2, \ldots, G_d]$. Then 

\[ E(K_d[G_1, G_2, \ldots, G_d]) \geq 2 \sum_{i=1}^{d-1} s_i +\sum_{i=1}^d E(G_i). \] 
In particular, if $d \geq 2$ then
\[ E(K_d[G_1, G_2, \ldots, G_d]) > E(K_d) +\sum_{i=1}^d E(G_i). \] 
\end{prop}
\begin{proof}
Let $\{\lambda_1, \lambda_2, \ldots, \lambda_d \}$ be the eigenvalues of $\overline{A}$ where $A$ and $\overline{A}$ are the matrices in Proposition \ref{prop:join_spetrum}, namely 
\[ \overline A= 
\begin{pmatrix}
r_1 & k_2 & \cdots & k_d \\
k_1 & r_2 & \cdots & k_d \\
\vdots  & \vdots  & \ddots & \vdots  \\
k_1 & k_2 & \cdots & r_d 
\end{pmatrix} .\] 
By the same argument as in Proposition \ref{prop:first_estimate}, we have 
\[ p_{\overline{A}}(-s_i)=-k_i \prod_{j \neq i} (s_j-s_i) .\] 
Because of the total ordering $s_1<s_2<\ldots<s_d$, we see that $p_{\overline{A}}(-s_i)p_{\overline{A}}(-s_{i+1})<0$ for $1 \leq i \leq d-1.$ By the mean value theorem, $p_{\overline{A}}(t)$ has a real root, say $\lambda_i$, in the interval $[-s_{i+1}, -s_i].$ In particular, $\lambda_i<0$ and $|\lambda_i| \geq s_i$ for $1 \leq i \leq d-1$.  We also note that 
\[ \sum_{i=1}^d \lambda_i=\Tr(\overline{A})=\sum_{i=1}^d r_i .\] 
Hence 
\[ \lambda_d=\sum_{i=1}^d r_i - \sum_{i=1}^{d-1} \lambda_i >0 .\] 
We then have 
\begin{align*}
E(K_d[G_1, G_2, \ldots, G_d])-\sum_{i=1}^d E(G_i) &= \sum_{i=1}^{d} |\lambda_i| -\sum_{i=1}^d r_i \\ 
&= \sum_{i=1}^{d-1} |\lambda_i|+ \left(\sum_{i=1}^d r_i - \sum_{i=1}^{d-1} \lambda_i \right)-\sum_{i=1}^d r_i \\
&=2 \sum_{i=1}^{d-1} |\lambda_i|  \geq 2 \sum_{i=1}^{d-1} s_i.
\end{align*}
Since $1 \leq s_1<s_2<\ldots<s_d$, the above inequality implies that 
\[E(K_d[G_1, G_2, \ldots, G_d])-\sum_{i=1}^d E(G_i)> 2(d-1)=E(K_d) .\] 
\end{proof}

\section*{Acknowledgments}
This work was supported by BrainsCAN at Western University through the Canada First Research Excellence Fund (CFREF), the NSF through a NeuroNex award (\#2015276), the Natural Sciences and Engineering Research Council of Canada (NSERC) grant R0370A01, and SPIRITS 2020 of Kyoto University. J.M.~gratefully acknowledges the Western University Faculty of Science Distinguished Professorship for 2020-2021. We acknowledge the support of the Western Academy for Advance Research during the year 2022-2023. We thank Jacqueline Doan for her discussions and support at the initial stage of this work. Parts of this article were written during the workshop ``Spectral graph and hypergraph theory: connections and applications" organized by the American Institute of Mathematics. T.T.N. would like to thank the organizers of this conference and the American Institute of Mathematics for the stimulating working environment and kind hospitality. Last but not least, we thank the referees and the Editorial Board for a number of suggestions which helped us improve the exposition of the previous version of our paper. 

\bibliographystyle{plain}
\bibliography{reference}

\begin{thebibliography}{10}

\bibitem{[KO3]}
Roberto~C. Budzinski, Tung~T. Nguyen, Jacqueline {\DJ}o\`an, J\'{a}n
  Min\'{a}\v{c}, Terrence~J. Sejnowski, and Lyle~E. Muller.
\newblock Geometry unites synchrony, chimeras, and waves in nonlinear
  oscillator networks.
\newblock {\em Chaos}, 32(3):{ }Paper No. 031104, 7, 2022.

\bibitem{[joined_union2]}
Domingos~M. Cardoso, Maria Aguieiras~A. de~Freitas, Enide~Andrade Martins, and
  Mar\'{i}a Robbiano.
\newblock Spectra of graphs obtained by a generalization of the join graph
  operation.
\newblock {\em Discrete Math.}, 313(5):733--741, 2013.

\bibitem{[joint_group_ring]}
Sunil~K Chebolu, Jonathan~L Merzel, J{\'a}n Min{\'a}{\v{c}}, Lyle Muller,
  Tung~T Nguyen, Federico~W Pasini, and Nguyen~Duy T{\^a}n.
\newblock On the joins of group rings.
\newblock {\em Journal of Pure and Applied Algebra}, 227(9):107377, 2023.

\bibitem{chudnovsky2024prime}
Maria Chudnovsky, Michal Cizek, Logan Crew, J{\'a}n Min{\'a}{\v{c}}, Tung~T
  Nguyen, Sophie Spirkl, and Nguy{\^e}n~Duy T{\^a}n.
\newblock On prime {Cayley} graphs.
\newblock {\em arXiv preprint arXiv:2401.06062}, 2024.

\bibitem{coulson1940calculation}
CA~Coulson.
\newblock On the calculation of the energy in unsaturated hydrocarbon
  molecules.
\newblock In {\em Mathematical Proceedings of the Cambridge Philosophical
  Society}, volume~36, pages 201--203. Cambridge University Press, 1940.

\bibitem{[CRS]}
Drago\v{s} Cvetkovi\'{c}, Peter Rowlinson, and Slobodan Simi\'{c}.
\newblock {\em An introduction to the theory of graph spectra}, volume~75 of
  {\em London Mathematical Society Student Texts}.
\newblock Cambridge University Press, Cambridge, 2010.

\bibitem{gutman1977acyclic}
Ivan Gutman.
\newblock Acyclic systems with extremal h{\"u}ckel $\pi$-electron energy.
\newblock {\em Theoretica chimica acta}, 45:79--87, 1977.

\bibitem{gutman2019graph}
Ivan Gutman and Boris Furtula.
\newblock Graph energies and their applications.
\newblock {\em Bulletin (Acad{\'e}mie serbe des sciences et des arts. Classe
  des sciences math{\'e}matiques et naturelles. Sciences math{\'e}matiques)},
  (44):29--45, 2019.

\bibitem{hoory2006expander}
Shlomo Hoory, Nathan Linial, and Avi Wigderson.
\newblock Expander graphs and their applications.
\newblock {\em Bulletin of the American Mathematical Society}, 43(4):439--561,
  2006.

\bibitem{[HJ]}
Roger~A. Horn and Charles~R. Johnson.
\newblock {\em Matrix analysis}.
\newblock Cambridge University Press, Cambridge, second edition, 2013.

\bibitem{jain2023composed}
Priya~B Jain, Tung~T Nguyen, J{\'a}n Min{\'a}{\v{c}}, Lyle~E Muller, and
  Roberto~C Budzinski.
\newblock Composed solutions of synchronized patterns in multiplex networks of
  kuramoto oscillators.
\newblock {\em Chaos: An Interdisciplinary Journal of Nonlinear Science},
  33(10), 2023.

\bibitem{mcclelland1971properties}
Bernard~J McClelland.
\newblock Properties of the latent roots of a matrix: the estimation of
  $\pi$-electron energies.
\newblock {\em The Journal of Chemical Physics}, 54(2):640--643, 1971.

\bibitem{[semimagic]}
Itiro Murase.
\newblock Semimagic squares and non-semisimple algebras.
\newblock {\em Amer. Math. Monthly}, 64:168--173, 1957.

\bibitem{[Murty]}
M.~Ram Murty.
\newblock Ramanujan graphs and zeta functions.
\newblock In {\em Algebra and number theory}, pages 269--280. Hindustan Book
  Agency, Delhi, 2005.

\bibitem{[KO4]}
Tung~T Nguyen, Roberto~C Budzinski, Jacqueline {\DH}o{\`a}n, Federico~W Pasini,
  J{\'a}n Min{\'a}{\v{c}}, and Lyle~E Muller.
\newblock Equilibria in {Kuramoto} oscillator networks: An algebraic approach.
\newblock {\em SIAM Journal on Applied Dynamical Systems}, 22(2):802--824,
  2023.

\bibitem{nguyen2022broadcasting}
Tung~T Nguyen, Roberto~C Budzinski, Federico~W Pasini, Robin Delabays, J{\'a}n
  Min{\'a}{\v{c}}, and Lyle~E Muller.
\newblock Broadcasting solutions on networked systems of phase oscillators.
\newblock {\em Chaos, Solitons \& Fractals}, 168:113166, 2023.

\bibitem{nguyen2024certain}
Tung~T Nguyen and Nguyen~Duy T{\^a}n.
\newblock On certain properties of the $ p $-unitary cayley graph over a finite
  ring.
\newblock {\em arXiv preprint arXiv:2403.05635}, 2024.

\bibitem{[CM1]}
Jacqueline {\DJ}o\`an, J\'{a}n Min\'{a}\v{c}, Lyle~E. Muller, Tung~T. Nguyen,
  and Federico~W. Pasini.
\newblock Joins of circulant matrices.
\newblock {\em Linear Algebra Appl.}, 650:190--209, 2022.

\bibitem{[joined_union]}
Dragan Stevanovi\'{c}.
\newblock Large sets of long distance equienergetic graphs.
\newblock {\em Ars Math. Contemp.}, 2(1):35--40, 2009.

\end{thebibliography}
\end{document}